\documentclass[12pt]{amsart}
\usepackage{amscd,amssymb,amsmath,multicol,float,scalefnt}
\restylefloat{table}
\newtheorem{thm}[equation]{Theorem}
\numberwithin{equation}{section}

\newtheorem{lem}[equation]{Lemma}

\newtheorem{notation}[equation]{Notation}

\begin{document}
\raggedbottom \voffset=-.7truein \hoffset=0truein \vsize=8truein
\hsize=6truein \textheight=8truein \textwidth=6truein
\baselineskip=18truept

\def\mapright#1{\ \smash{\mathop{\longrightarrow}\limits^{#1}}\ }
\def\mapleft#1{\smash{\mathop{\longleftarrow}\limits^{#1}}}
\def\mapup#1{\Big\uparrow\rlap{$\vcenter {\hbox {$#1$}}$}}
\def\mapdown#1{\Big\downarrow\rlap{$\vcenter {\hbox {$\ssize{#1}$}}$}}
\def\mapne#1{\nearrow\rlap{$\vcenter {\hbox {$#1$}}$}}
\def\mapse#1{\searrow\rlap{$\vcenter {\hbox {$\ssize{#1}$}}$}}
\def\mapr#1{\smash{\mathop{\rightarrow}\limits^{#1}}}
\def\ss{\smallskip}
\def\vp{v_1^{-1}\pi}
\def\at{{\widetilde\alpha}}
\def\sm{\wedge}
\def\la{\langle}
\def\ra{\rangle}
\def\on{\operatorname}
\def\spin{\on{Spin}}
\def\kbar{{\overline k}}
\def\qed{\quad\rule{8pt}{8pt}\bigskip}
\def\ssize{\scriptstyle}
\def\a{\alpha}
\def\bz{{\Bbb Z}}
\def\im{\on{im}}
\def\ct{\widetilde{C}}
\def\ext{\on{Ext}}
\def\sq{\on{Sq}}
\def\eps{\epsilon}
\def\ar#1{\stackrel {#1}{\rightarrow}}
\def\br{{\Bbb R}}
\def\bC{{\bold C}}
\def\bA{{\bold A}}
\def\bB{{\bold B}}
\def\bD{{\bold D}}
\def\bh{{\bold H}}
\def\bQ{{\bold Q}}
\def\bP{{\bold P}}
\def\bx{{\bold x}}
\def\bo{{\bold{bo}}}
\def\si{\sigma}
\def\Ebar{{\overline E}}
\def\dbar{{\overline d}}
\def\Sum{\sum}
\def\tfrac{\textstyle\frac}
\def\tb{\textstyle\binom}
\def\Si{\Sigma}
\def\w{\wedge}
\def\equ{\begin{equation}}
\def\b{\beta}
\def\G{\Gamma}
\def\g{\gamma}
\def\k{\kappa}
\def\psit{\widetilde{\Psi}}
\def\tht{\widetilde{\Theta}}
\def\psiu{{\underline{\Psi}}}
\def\thu{{\underline{\Theta}}}
\def\aee{A_{\text{ee}}}
\def\aeo{A_{\text{eo}}}
\def\aoo{A_{\text{oo}}}
\def\aoe{A_{\text{oe}}}
\def\fbar{{\overline f}}
\def\endeq{\end{equation}}
\def\sn{S^{2n+1}}
\def\zp{\bold Z_p}
\def\A{{\cal A}}
\def\P{{\mathcal P}}
\def\cj{{\cal J}}
\def\zt{{\Bbb Z}_2}
\def\bs{{\bold s}}
\def\bof{{\bold f}}
\def\bq{{\bold Q}}
\def\be{{\bold e}}
\def\Hom{\on{Hom}}
\def\ker{\on{ker}}
\def\coker{\on{coker}}
\def\da{\downarrow}
\def\colim{\operatornamewithlimits{colim}}
\def\zphat{\bz_2^\wedge}
\def\io{\iota}
\def\Om{\Omega}
\def\Prod{\prod}
\def\e{{\cal E}}
\def\exp{\on{exp}}
\def\kbar{{\overline w}}
\def\xbar{{\overline x}}
\def\ybar{{\overline y}}
\def\zbar{{\overline z}}
\def\ebar{{\overline e}}
\def\nbar{{\overline n}}
\def\rbar{{\overline r}}
\def\et{{\widetilde E}}
\def\ni{\noindent}
\def\coef{\on{coef}}
\def\den{\on{den}}
\def\lcm{\on{l.c.m.}}
\def\vi{v_1^{-1}}
\def\ot{\otimes}
\def\psibar{{\overline\psi}}
\def\mhat{{\hat m}}
\def\exc{\on{exc}}
\def\ms{\medskip}
\def\ehat{{\hat e}}
\def\etao{{\eta_{\text{od}}}}
\def\etae{{\eta_{\text{ev}}}}
\def\dirlim{\operatornamewithlimits{dirlim}}
\def\gt{\widetilde{L}}
\def\lt{\widetilde{\lambda}}
\def\st{\widetilde{s}}
\def\ft{\widetilde{f}}
\def\sgd{\on{sgd}}
\def\lfl{\lfloor}
\def\rfl{\rfloor}
\def\ord{\on{ord}}
\def\gd{{\on{gd}}}
\def\rk{{{\on{rk}}_2}}
\def\nbar{{\overline{n}}}
\def\lg{{\on{lg}}}
\def\cR{\mathcal{R}}
\def\cT{\mathcal{T}}
\def\N{{\Bbb N}}
\def\Z{{\Bbb Z}}
\def\Q{{\Bbb Q}}
\def\R{{\Bbb R}}
\def\C{{\Bbb C}}
\def\l{\left}
\def\r{\right}
\def\mo{\on{mod}}
\def\vexp{v_1^{-1}\exp}
\def\notimm{\not\subseteq}
\def\Remark{\noindent{\it  Remark}}

\def\ss{\smallskip}
\def\ssum{\sum\limits}
\def\dsum{\displaystyle\sum}
\def\la{\langle}
\def\ra{\rangle}
\def\on{\operatorname}
\def\o{\on{o}}
\def\U{\on{U}}
\def\lg{\on{lg}}
\def\a{\alpha}
\def\bz{{\Bbb Z}}
\def\eps{\varepsilon}
\def\br{{\Bbb R}}
\def\bc{{\bold C}}
\def\bN{{\bold N}}
\def\nut{\widetilde{\nu}}
\def\tfrac{\textstyle\frac}
\def\b{\beta}
\def\G{\Gamma}
\def\g{\gamma}
\def\zt{{\Bbb Z}_2}
\def\zth{{\bold Z}_2^\wedge}
\def\bs{{\bold s}}
\def\bx{{\bold x}}
\def\bof{{\bold f}}
\def\bq{{\bold Q}}
\def\be{{\bold e}}
\def\lline{\rule{.6in}{.6pt}}
\def\xb{{\overline x}}
\def\xbar{{\overline x}}
\def\ybar{{\overline y}}
\def\zbar{{\overline z}}
\def\ebar{{\overline \be}}
\def\nbar{{\overline n}}
\def\rbar{{\overline r}}
\def\Mbar{{\overline M}}
\def\et{{\widetilde e}}
\def\ni{\noindent}
\def\ms{\medskip}
\def\ehat{{\hat e}}
\def\xhat{{\widehat x}}
\def\nbar{{\overline{n}}}
\def\minp{\min\nolimits'}
\def\N{{\Bbb N}}
\def\Z{{\Bbb Z}}
\def\Q{{\Bbb Q}}
\def\R{{\Bbb R}}
\def\C{{\Bbb C}}
\def\el{\ell}
\def\TC{\on{TC}}
\def\dstyle{\displaystyle}
\def\ds{\dstyle}
\def\Remark{\noindent{\it  Remark}}
\title
{Topological complexity of some planar polygon spaces}
\author{Donald M. Davis}
\address{Department of Mathematics, Lehigh University\\Bethlehem, PA 18015, USA}
\email{dmd1@lehigh.edu}
\date{July 6, 2015}


\maketitle
\begin{abstract} Let $\Mbar_{n,r}$ denote the space of isometry classes of $n$-gons in the plane with one side of length $r$ and all others of length 1, and assume that $n-r$ is not an odd integer. Using known results about the mod-2 cohomology ring, we prove that its topological complexity satisfies $\TC(\Mbar_{n,r})\ge 2n-6$. Since $\Mbar_{n,r}$ is an $(n-3)$-manifold, $\TC(\Mbar_{n,r})\le 2n-5$. So our result is within 1 of optimal.
 \end{abstract}

\section{Statement of results}\label{intro}
The topological complexity, $\TC(X)$, of a topological space $X$ is, roughly, the number of rules required to specify how to move between any two points of $X$. A ``rule'' must be such that the choice of path varies continuously with the choice of endpoints. (See \cite[\S4]{F}.) We study $\TC(X)$ where $X=\Mbar_{n,r}$ is the space of isometry classes of $n$-gons in the plane with one side of length $r$ and all others of length 1. (See, e.g., \cite[\S9]{HK}.) Here $r$ is a real number satisfying $0<r<n-1$, and $n\ge4$. Thus
$$\Mbar_{n,r}=\{(z_1,\ldots,z_n)\in (S^1)^n:z_1+\cdots+z_{n-1}+rz_n=0\}/O(2).$$
If we think of the sides of the polygon as linked arms of a robot, we might prefer the space $M_{n,r}$, in which we identify only under rotation, and not also under reflection. However, the cohomology algebra of $\Mbar_{n,r}$ is better understood than that of $M_{n,r}$, leading to better bounds on TC.

If $r$ is a positive real number, then $\Mbar_{n,r}$ is a connected $(n-3)$-manifold unless $n-r$ is an odd integer (e.g., \cite[p.314]{HK} or \cite[p.2]{KK}), and hence satisfies \begin{equation}\label{bound}\TC(\Mbar_{n,r})\le 2n-5\end{equation} by \cite[Cor 4.15]{F}.\footnote{If $n-r$ is an odd integer, $\Mbar_{n,r}$ is often not a manifold but still satisfies $\TC(\Mbar_{n,r})\le 2n-5$, by \cite[Theorem 4]{F2}. However, its cohomology algebra is not so well understood in this case, and so we do not study it here.} By \cite[6.2]{Hb}, if, for an integer $k$, $n-2k-1<r<n-2k+1$, then $\Mbar_{n,r}$ is diffeomorphic to $\Mbar_{n,n-2k}$, and so we restrict our discussion to the latter spaces.
In this paper, we obtain the following strong lower bound for $\TC(\Mbar_{n,n-2k})$.

\begin{thm}\label{thm1} If $2<2k<n$, then $\TC(\Mbar_{n,n-2k})\ge 2n-6$.\end{thm}

This result is within 1 of being optimal, using (\ref{bound}).
 The case $k=1$ is special, as $\Mbar_{n,n-2}$ is homeomorphic to real projective space $RP^{n-3}$, for which the topological complexity agrees with the immersion dimension, a much-studied concept, but not yet completely determined.
See, e.g., \cite{FTY}, \cite{D}, or \cite{H}. In fact, there are often large gaps between the known upper and lower bounds for $\TC(RP^n)$.(\cite{Da})

The proof of Theorem \ref{thm1} relies on the  mod 2 cohomology ring $H^*(\Mbar_{n,r};\zt)$, first described in \cite{HK}.
Throughout the paper, all cohomology groups have coefficients in $\zt$, and all congruences are mod 2, unless specifically stated to the contrary.
 To prove Theorem \ref{thm1}, we will find  $2n-7$ classes $y_i\in H^1(\Mbar_{n,n-2k})$ such that $\prod(y_i\ot1+1\ot y_i)\ne0$ in $H^{n-3}(\Mbar_{n,n-2k})\ot H^{n-4}(\Mbar_{n,n-2k})$.
This implies the theorem by the basic result that if in $H^*(X\times X)$ there is an $m$-fold nonzero product of classes of the form $y_i\otimes1+1\otimes y_i$,
 then $\TC(X)\ge m+1$.(\cite[Cor 4.40]{F}) We show at the end of the paper that our cohomology result for $\Mbar_{n,n-2k}$ is optimal, in that $(2n-6)$-fold products of $(y_1\ot1+1\ot y_i)$ are always 0. Thus we will have proved the following result. (See \cite{F2} for the definition.)
 \begin{thm}\label{thm3} If $2<2k<n$, the zero-divisors-cup-length of $H^*(\Mbar_{n,n-2k})$ equals $2n-7$.\end{thm}

\section{Proof}\label{pfsec}
In this section we prove Theorems \ref{thm1} and \ref{thm3}.
We begin by stating our  interpretation of the cohomology ring $H^*(\Mbar_{n,n-2k})$.
\begin{thm} \label{cohthm} Let $k\ge1$ and $n>2k$.
\begin{enumerate}
\item The algebra $H^*(\Mbar_{n,n-2k})$ is generated by  classes $R,V_1,\ldots,V_{n-1}$ in $H^1(\Mbar_{n,n-2k})$.
\item The product of $k$ distinct $V_i$'s is 0.
\item If $d\le n-3$ and $S\subset\{1,\ldots,n-1\}$ has $|S|<k$,
then all monomials $\dstyle{R^{e_0}\prod_{i\in S}V_i^{e_i}}$ with $e_i>0$ for $i\in S$ and $\dstyle\sum_{i\ge0} e_i=d$ are equal. We denote this class by $T_{S,d}$. This includes the class  $T_{\emptyset,d}=R^d$.
\item  For every subset $L$ of $\{1,\ldots,n-1\}$ with $n-k\le|L|\le d+1$, there is a relation $\cR_{L,d}$ which says
$$\sum_{S\subset L}T_{S,d}=0.$$
    These are the only relations, in addition to those previously described.
    \end{enumerate}
\end{thm}

\begin{proof} In \cite[Theorem 1]{KK}, the more general result proved in \cite[Corollary 9.2]{HK} is applied to $\Mbar_{n,n-2k}$.
The first three parts of our theorem are immediate from the result stated there, although our $T_{S,d}$ notation is new. The relations stated in \cite{KK}
are in the form of an ideal, whereas we prefer to make a listing of a basic set of relations. The result of \cite{KK} says that the relations
in $H^*(\Mbar_{n,n-2k})$ comprise the ideal generated by
\begin{equation}\label{Treln}\sum_{S\subset L}T_{S,|L|-1} \text{ for }L\subset\{1,\ldots,n-1\}\text{ with }n-k\le|L|\le n-2.\end{equation}
Multiplying this relation by $R^t$ gives a relation $\dstyle{\sum_{S\subset L}T_{S,|L|-1+t}}$. This yields, in degree $d$, exactly all of our claimed relations. Additional relations in the ideal can be obtained by multiplying (\ref{Treln}) by $V_\ell$. If $\ell\not\in L$, this equals
our $\cR_{S\cup\{\ell\},|L|}-\cR_{S,|L|}$, while if $\ell\in L$, it equals 0.
\end{proof}

Most of our proofs also utilize the following key result, which was proved as \cite[Theorem B]{KK}.

\begin{lem}\label{KKB} There is an isomorphism  $\phi_1:H^{n-3}(\Mbar_{n,n-2k})\to\zt$ satisfying $\phi_1(T_{S,n-3})=\binom{n-2-|S|}{k-1-|S|}$.
\end{lem}

We begin our work with a useful lemma.
\begin{lem}\label{n-4lem} There is a homomorphism
$$\phi_2:H^{n-4}(\Mbar_{n,n-2k})\to\zt$$
satisfying $\phi_2(T_{S,n-4})=\binom{n-2-|S|}{k-1-|S|}$.
\end{lem}
\begin{proof} We must show that $\phi_2$ sends each of the relations $\cR_{L,n-4}$ to 0. If $|L|=\ell$, then
$$\phi_2(\cR_{L,n-4})=\sum_{i=0}^{k-1}{\tbinom{\ell}i\tbinom{n-2-i}{k-1-i}}=\sum_i\tbinom{\ell}i\tbinom{-n+k}{k-1-i}=\tbinom{\ell-n+k}{k-1}.$$
Since $n-k\le \ell\le n-3$, we have $0\le\ell-n+k\le k-3$, and so $\binom{\ell-n+k}{k-1}=0$.
\end{proof}

To prove Theorem \ref{thm1}, we will find  $2n-7$ classes $y_i\in H^1(\Mbar_{n,n-2k})$ such that $\prod(y_i\ot1+1\ot y_i)\ne0$ in $H^{n-3}(\Mbar_{n,n-2k})\ot H^{n-4}(\Mbar_{n,n-2k})$.
There will be four cases, Theorems \ref{case1}, \ref{case3}, \ref{case4}, and \ref{case5}. All of them use the following notation, which pervades the rest of the paper.
\begin{notation}\label{nk}
Let $t\ge0$ and $k=2^t+k_0$, $1\le k_0\le 2^t$, and $n=k+1+2^tB+D$ with $0\le D<2^t$ and $B\ge1$. Let $C=k_0+D-1$. Then $n=2^t(B+1)+C+2$.
\end{notation}
\noindent Every pair $(k,n)$ with $k\ge2$ and $n>2k$ yields unique values of $t$, $k_0$, $B$, and $D$.

\begin{thm}\label{case1} Let $B$ be odd and
\begin{eqnarray}\label{phiphi}P&=&(V_1\ot1+1\ot V_1)^{2^t(B+1)-1}\cdot\Prod^{C}(V_i\ot1+1\ot V_i)\\
&&\quad \cdot\Prod^{C}(V_i\ot1+1\ot V_i)^2\cdot(R\ot1+1\ot R)^{2^t(B+1)-C-2}.\nonumber\end{eqnarray}
If $P_1$ denotes the component of $P$ in $H^{2^t(B+1)+C-1}(\Mbar_{n,n-2k})\ot H^{2^t(B+1)+C-2}(\Mbar_{n,n-2k})$, then
$(\phi_1\ot\phi_2)(P_1)\ne0\in \zt.$
\end{thm}
The product notation here, which will be continued throughout the paper, means a product of $C$ distinct factors with subscripts distinct from other subscripts involved elsewhere in the expression. Since $P$ has $2n-7$ factors, Theorem \ref{case1} implies Theorem \ref{thm1} when $B$ is odd.

\begin{proof}
Since $B$ is odd,  the third case of Lemma \ref{techlem} applies.  Note that $(V_1\ot1+1\ot V_i)^2=V_1^2\ot1+1\ot V_i^2$, and that there are $2C$ factors $F$ of this form or $V_i\ot1+1\ot V_i$ in the middle of $P$.
When $P$ is expanded, the only terms $\cT$ for which $(\phi_1\ot\phi_2)(\cT)$ might possibly be nonzero are those with exactly $C$ of these factors $F$ on each side of $\ot$ accompanied by a nontrivial contribution from the $V_1$-part. (This uses Lemmas \ref{KKB}, \ref{n-4lem}, and \ref{techlem}.) Such a term $\cT$ which contains $j$ of the $V_i$'s ($i>1$) (and $(C-j)$ $V_i^2$'s) on the left side of $\ot$ will be of the form
\begin{equation}\label{is}\tbinom{2^t(B+1)-1}e\tbinom{2^t(B+1)-C-2}{2^t(B+1)-C-1+j-e}V_1^eV_{i_1}\cdots V_{i_j}V_{i_{j+1}}^2\cdots V_{i_{C}}^2R^{2^t(B+1)-C-1+j-e}\ot Q,\end{equation}
where $Q$ is the complementary factor. Here we must have $0<e<2^t(B+1)-1$, in order that there are $C+1$ distinct $V_i$ factors on both sides of $\ot$. For this choice of $(i_1,\ldots,i_{C})$, let $W$ denote the sum of all such terms as $e$ varies, with $i_1,\ldots,i_C$ fixed. Then
\begin{eqnarray}\nonumber(\phi_1\ot\phi_2)(W)&=&\sum_{e=1}^{2^t(B+1)-2}\tbinom{2^t(B+1)-1}e\tbinom{2^t(B+1)-C-2}{2^t(B+1)-C-1+j-e}\\
&\equiv&\tbinom{2^{t+1}(B+1)-C-3}{2^t(B+1)-C-1+j}+\tbinom{2^t(B+1)-C-2}{2^t(B+1)-C-1+j}+\tbinom{2^t(B+1)-C-2}{-C+j}.\label{three}\end{eqnarray}
The first of the three terms in the last line is what the sum would have been if the terms with $e=0$ and $e=2^t(B+1)-1$ were included, while the other two terms are the two omitted terms. Mod 2, the first binomial coefficient is 0 by Lemma \ref{bclem}, since  the case with $B=1$  and $C=2^{t+1}-2$ does not satisfy $n>2k$. The second binomial coefficient in (\ref{three}) is 0 because its bottom part is greater than its top, and the third is 0 unless $j=C$. Thus there is a unique\footnote{The uniqueness refers to the choice of which squared terms appear on the left side of $\ot$ in (\ref{is}), given the choice of $i$'s in (\ref{phiphi}). The choice of which values of $i$ occur   in (\ref{phiphi}) is arbitrary, and far from unique.}  $W$, namely
$$W=\sum_{e=1}^{2^t(B+1)-2}c_eV_1^e\bigl(\prod^C V_i\bigr) R^{2^t(B+1)-e-1}\ot  V_1^{2^t(B+1)-1-e}\bigl(\prod^C V_{i}^2\bigr)R^{e-C-1},$$
with $c_e=\tbinom{2^t(B+1)-1}e\tbinom{2^t(B+1)-C-2}{2^t(B+1)-1-e}$,
for which $\phi(W)=1$, establishing the claim in this case.
\end{proof}

The following lemmas were used above.
\begin{lem}\label{techlem} In the notation of \ref{nk},
$$\binom{n-2-i}{k-1-i}\equiv\begin{cases}1&i=C+1\\
0&k_0\le i\le C\\
0&0\le i\le C\text{ if $B$ is odd.}\end{cases}$$
\end{lem}
\begin{proof} We have $\binom{n-2-i}{k-1-i}=\binom{n-2-i}{2^tB+D}$ with $0\le D<2^t$. If $i=C+1$, then $n-2-i=2^tB+2^t-1$, and so the binomial coefficient is odd by Lucas's Theorem, which we will often use without comment. Decreasing $i$ by $1,\ldots,D$ increases the top of the binomial coefficient by that amount, yielding $\binom{2^t(B+1)+j}{2^tB+D}$ with $0\le j<D$. Such a binomial coefficient is even.
If $B$ is odd, decreasing $i$ even more will leave the binomial coefficient even, as it will be either $\binom{B+1}B\binom jD$ with $j\ge D$ or $\binom{B+2}B\binom jD$ with $j\le C-2^t <D$.
\end{proof}

\begin{lem}\label{bclem} If $0\le C\le 2^{t+1}-2$ and $0\le j\le C$, then, mod 2,
$$\binom{2^{t+1}(B+1)-C-3}{2^t(B+1)-C-1+j}\equiv\begin{cases}1&\text{if $B$ is a $2$-power and $C=2^{t+1}-2$}\\
0&\text{otherwise.}\end{cases}$$
\end{lem}
\begin{proof} If $C=2^{t+1}-2$, then the binomial coefficient is $\binom{2^{t+1}B-1}{2^tB+\Delta}$ with $|\Delta|<2^t$. For $\Delta=0$ this is odd iff $B$ is a 2-power, as is easily seen using Lucas's Theorem. If the bottom part of the binomial coefficient is changed from $\Delta=0$ by an amount less than $2^t$, the binomial coefficient is multiplied by $p/q$ with $p$ and $q$ equally 2-divisible. If $C=2^{t+1}-3$, then the binomial coefficient is of the form $\binom{2^{t+1}B}{2^tB+\Delta}$ with $|\Delta|<2^t$. This is even for all $B$, similarly to the previous case. For smaller values of $C$, the result follows by induction on (decreasing) $C$, using Pascal's formula. Here it is perhaps more convenient to think of the binomial coefficient as $\binom{2^{t+1}(B+1)-C-3}{2^t(B+1)-j-2}$.
\end{proof}

The case in which $D=0$ and $B$ is even is special because then $(\phi_1\ot\phi_2)(M)=1$ for every monomial $M$ in $H^{n-3}(\Mbar_{n,n-2k})\ot H^{n-4}(\Mbar_{n,n-2k})$, and so for any appropriate product $P$, we have $(\phi_1\ot\phi_2)(P)=\binom{2n-7}{n-3}=0$ (unless $n-3$ is a 2-power.) So we modify $\phi_2$.

\begin{thm}\label{case3} In the notation of \ref{nk}, let $B$ be even and $D=0$. There is a homomorphism
$$\phi_3:H^{n-4}(\Mbar_{n,n-2k})\to\zt$$
defined by $$\phi_3(T_{S,n-4})=\begin{cases}1&|S|<k-1\\ 0&|S|=k-1.\end{cases}$$
If
$$P=(V_1\ot1+1\ot V_1)^{n-3}\cdot\prod^{k-2}(V_i\ot1+1\ot V_i)\cdot(R\ot1+1\ot R)^{n-k-2},$$
then $(\phi_1\ot\phi_3)(P)=1\in\zt$.\end{thm}
\begin{proof} To prove that $\phi_3$ is well-defined, we must show that for $n-k\le\ell\le n-3$, we have $\ds\sum_{i=0}^{k-2}\tbinom\ell i\equiv 0$. Then $2^tB<\ell\le2^tB+k-2$. Since $B$ is even and $k\le2^{t+1}$, the $2^tB$ does not affect the binomial coefficient mod 2, and the sum becomes $\ds\sum_{i=0}^{k-2}\tbinom\ell i$ for $0<\ell\le k-2$, and this equals $2^\ell$.

Since $\phi_1(M)=1$ for every monomial in $H^{n-3}(\Mbar_{n,n-2k})$, $(\phi_1\ot\phi_3)(P)$ equals the sum of coefficients in
$$(1+V_1)^{n-3}\cdot\prod_{i=2}^{k-1}(1+V_i)\cdot(1+R)^{n-2-k}$$
of all  monomials of degree $n-4$ which are not divisible by $V_1\cdots V_{k-1}$. This equals $S_1-(S_2-S_3)$, where $S_1$ is the sum of all coefficients in degree $n-4$, $S_2$ is the sum of coefficients of terms divisible by $V_2\cdots V_{k-1}$, and $S_3$ is the sum of coefficients of terms divisible by $V_2\cdots V_{k-1}$ but not also by $V_1$. Then $S_1=\binom{2n-7}{n-4}\equiv0$ since $n-3$ cannot be a 2-power here. Also $S_2=\binom{2n-5-k}{n-4-(k-2)}=\binom{2^{t+1}B+k-3}{2^tB-1}\equiv0$ since $k\le2^{t+1}$. Finally for $S_3$ the only monomial is $V_2\cdots V_{k-1}R^{n-2-k}$,
so $S_3=1$.
\end{proof}

Let $\lg(-)=[\log_2(-)]$.
\begin{thm}\label{case4} Theorem \ref{case1} is true if $B$ is even and $C-2^{\lg(C)}<2^{1+\lg D}$.\end{thm}
The first few cases of this hypothesis are ($D=1$ and $C\in\{2^e,2^e+1\}$) and ($D\in\{2,3\}$ and $C\in\{2^e,2^e+1,2^e+2,2^e+3\}$).
\begin{proof} We consider first the portion $P_2$ of the expansion of $P$ which has $V_1^{2^t(B+1)-1}$ on the left side of $\ot$. If $j$ (resp.~$g$) denotes the number of other $V_i$'s (resp.~$V_i^2$'s) on the left side of $\ot$, then $(\phi_1\ot\phi_2)(P_2)$ equals
\begin{equation}\label{sum1}\sum_{j,g=0}^C\tbinom Cj\tbinom Cg\tbinom{2^t(B+1)-1+C-j-g}{2^tB+C+1-k_0}\tbinom{2^t(B+1)-(C-j-g)}{2^tB+C+1-k_0}\tbinom{2^t(B+1)-C-2}{C-j-2g}.\end{equation}
The third and fourth factors here are from $\phi_1(-)$ and $\phi_2(-)$, which satisfy
$$\phi_1(T_{S,n-3})=\phi_2(T_{S,n-4})=\tbinom{2^t(B+1)+C-|S|}{2^t+k_0-1-|S|}=\tbinom{2^t(B+1)+C-|S|}{2^tB+C+1-k_0}.$$
These two factors in our sum are of the form $\tbinom{2^t(B+1)-1+\Delta}{2^tB+D}\tbinom{2^t(B+1)-\Delta}{2^tB+D}$ with $-C\le\Delta\le C$ and $1\le D\le 2^t-1$.
In positions less than $2^t$, the top parts of these two binomial coefficients differ in every position, and so due to any position where $D$ has a 1, one of the factors will be even.
Thus (\ref{sum1}) is 0 in $\zt$. A similar argument works for the portion of the sum in which $V_1^{2^t(B+1)-1}$ is on the right side of $\ot$.

Arguing similarly to (\ref{three}), it remains to show that the following sum is 1 mod 2.
\begin{eqnarray}\label{fh} &&\sum_{j,g=0}^C\tbinom Cj\tbinom Cg\tbinom{2^t(B+1)+C-j-g-1}{2^tB+D}\tbinom{2^t(B+1)-1-C+j+g}{2^tB+D}\\
&&\quad\cdot\bigl(\tbinom{2^{t+1}(B+1)-C-3}{2^t(B+1)+C-1-j-2g}+\tbinom{2^t(B+1)-C-2}{2^t(B+1)+C-1-j-2g}+\tbinom{2^t(B+1)-C-2}{C-j-2g}\bigr).\label{sum2}\end{eqnarray}

Let $\ell=\lg(D)$. Note that $t\ge\ell+1$. Keep in mind that $B$ is even.
It is easy to check that there is a nonzero summand due to the third term of (\ref{sum2}) if $(j,g)=(C,0)$, and one due to the first term of (\ref{sum2}) if $t=\ell+1$, $D=2^{\ell+1}-1$, $j+g=C=2^{\ell+2}-2$ with $j$ even and $0\le j\le C$, and $B$ is a 2-power. The proof will be completed by showing that other terms are nonzero iff $C=2^{\ell+2}-1$, $t\ge\ell+2$, and $|C-j-g|=2^{\ell+1}$. The result will follow, as the total number of nonzero terms is odd in any case.

It is also easy to check that the  terms of the third type give nonzero summands. For example, let $\ell=2$, so we have $D\in\{4,5,6,7\}$, $C=15$, and $j+g=7$ or 23. Then $7\le j+2g\le14$ in the first case, and, since
$j,g\le 15$, we have $31\le j+2g\le 38$ in the second case. Also $t\ge4$ and $B$ is even. The latter two factors in (\ref{fh}) are $\binom{2^t(B+1)-1\pm8}{2^tB+D}\equiv1$.
Of the three terms  in (\ref{sum2}), the first will be even since it is either $\binom{2^{t+1}B+\a}{2^tB+\b}$ with $0<\a<2^{t+1}$ and $0<\b<2^t$, or $\binom{32B+14}{16B-\gamma}$ with $1\le\gamma\le8$. When $j+g=7$, the second summand in (\ref{sum2}) has bottom greater than top, while the third is of the form $\binom{16A+15}b$ with $1\le b\le8$, hence is odd. When $j+g=23$, the second summand is of the form $\binom{16A+15}{16A+c}$ with $8\le c\le 15$, while the third has its bottom part negative.

Now we show that all other terms in (\ref{fh})-(\ref{sum2}) are 0. Let $t\ge\ell+1$, $2^e\le C<2^{e+1}$ with $e\le t$ and $C-2^e<2^{\ell+1}$.
We will show
\begin{enumerate} \item If $$P_0=\tbinom Cj\tbinom Cg\tbinom{2^t(B+1)+C-j-g-1}{2^tB+D}\tbinom{2^t(B+1)-1-C+j+g}{2^tB+D}$$
then $P_0$ is odd iff $\binom Cj$ and $\binom C g$ are odd and $C-j-g\equiv0$ mod $2^{\ell+1}$ and $|C-j-g|<2^t$.
\item In the cases just noted where $P_0$ is odd,
\begin{equation}\label{3b}\tbinom{2^{t+1}(B+1)-C-3}{2^t(B+1)+C-1-j-2g}\equiv\tbinom{2^t(B+1)-C-2}{2^t(B+1)+C-1-j-2g}\equiv\tbinom{2^t(B+1)-C-2}{C-j-2g}\equiv0\end{equation}
except in the cases noted in the paragraph following (\ref{fh})-(\ref{sum2}).
\end{enumerate}

To prove (1), let $E=C-j-g$. Then $-C\le E\le C$, but by symmetry, it suffices to consider $0\le E\le C<2^{t+1}$. It is easy to see that $E-1$ and $-E-1$ both have a 1 in the $2^\ell$-position iff $E\equiv0$ mod $2^{\ell+1}$. Since $D$ has a 1 in the $2^\ell$-position, $P_0$ is even unless $E\equiv0$ mod $2^{\ell+1}$.  Letting $E=2^{\ell+1}E'$, and removing the lower parts of the binomial coefficients, we need for
$$\binom{2^t(B+1)+2^{\ell+1}(E'-1)}{2^tB}\binom{2^t(B+1)-2^{\ell+1}(E'+1)}{2^tB}$$
to be odd. If $2^{\ell+1}E'\ge 2^t$, the second binomial coefficient is 0. Otherwise, $2^{\ell+1}(E'+1)\le 2^t$, and then both binomial coefficients are odd,

For (2), we first study how the middle coefficient in (\ref{3b}) can be odd. If $t=\ell+1$, then $C-j-g=0$, and the binomial coefficient is 0 since its bottom part is greater than the top. Now assume $t\ge\ell+2$. Let $C-j-g=-2^{\ell+1}K$. The binomial coefficient becomes $\ds\binom{2^t(B+1)-2-C}{2^{\ell+2}K-j-1}$. For this to be odd, in positions $<2^{\ell+2}$ the 1's in (the binary expansion of) $(C+1$) must be contained in those of $j$. For $\binom Cj$ to be odd, the 1's of $j$ must be contained in those of $C$. The only way that the 1's of $(C+1)$ can be contained in those of $C$ in these positions is if $C+1\equiv0$ mod $2^{\ell+2}$. Since $C-2^{\lg C}<2^{1+\ell}$, the only such $C$ is $2^{\ell+2}-1$. Since $-2^{\ell+2}<C-j-g<0$, we must have $C-j-g=-2^{\ell+1}$. Thus the only ways the middle coefficient of (\ref{3b}) can yield a nonzero value are those listed earlier.

The third coefficient in (\ref{3b}) is handled similarly. If $t=\ell+1$, then $C-j-g=0$ and $g=0$, yielding a claimed condition. Now assume $t\ge\ell+2$. Let $C-j-g=2^{\ell+1}K$. The binomial coefficient becomes $\ds\binom{2^t(B+1)-C-2}{2^t(B+1)-2^{\ell+2}K-j-2}$. For this and $\binom Cj$ to both be odd, either $C+1\equiv0$ mod $2^{\ell+2}$ or $j\equiv C$ mod $2^{\ell+2}$. The former condition reduces to $C=2^{\ell+2}-1$, $C-j-g=2^{\ell+1}$ similarly to the previous case. For the latter, if $C=j$, then $g=0$ and we obtain one of the claimed conditions. Otherwise, write $C=2^e+\Delta$ with $0\le\Delta<2^{\min(e,\ell+1)}$. Then we must have $e\ge\ell+2$ and $j=\Delta$, and, since $g\equiv 0$ mod $2^{\ell+1}$ and $\binom Cg$ is odd, we must have $g=0$ or $2^e$, neither of which make $\binom{2^t(B+1)-C-2}{C-j-2g}$ odd, since $e<t$ .

For the first coefficient in (\ref{3b}), we first consider the situation when $t=\ell+1$. Then $C-j-g=0$ and the coefficient equals $\ds\binom{2^{\ell+2}(B+1)-C-3}{2^{\ell+1}(B+1)-2-j}$.
For this and $\binom Cj$ to both be odd, we must have $C\equiv-\eps$ mod $2^{\ell+1}$ with $\eps\in\{1,2\}$, and $j$ even. If $C=2^{\ell+1}-\eps$, then the coefficient is $\binom{2^{\ell+2}B+\a}{2^{\ell+1}B+\b}$, with $0\le\a,\b<2^{\ell+1}$, and thus is even, due to $\binom{2B}B$. If $C=2^{\ell+2}-2$ (its largest possible value) and $B$ is not a 2-power, the coefficient can be written as $\binom{2^{\nu+1}(2A+1)-1}{2^\nu(2A+1)+\Delta}$ with $A>0$ and $0\le\Delta<2^\nu$, which is even since it splits as $\binom{2^{\nu+2}A}{2^{\nu+1}A}\binom{2^{\nu+1}-1}{2^\nu+\Delta}$.

If $t\ge\ell+2$,
let $C-j-g=2^{\ell+1}K$ and write the coefficient as $\binom{2^{t+1}(B+1)-C-3}{2^t(B+1)-2-2^{\ell+2}K-j}$. For both this and $\binom Cj$ to be odd, we must have $C\equiv-1$ or $-2$ mod $2^{\ell+2}$ and $j$ even. Since $C-2^{\lg C}<2^{\ell+1}$, this implies $C=2^{\ell+2}-1$ or $2^{\ell+2}-2$. The top part of the binomial coefficient splits as $2^{t+1}B+(2^{t+1}-2^{\ell+2}-\eps)$. Since $|2^{\ell+1}K|\le C$, then $|2^{\ell+1}K|\le2^{\ell+1}$. Thus
the bottom of the binomial coefficient is $2^tB+\a$ with
$$2^t-2^{\ell+3}\le\a\le 2^t+2^{\ell+2}-2.$$
Since $B$ is even, the binomial coefficient, mod 2, splits as $\binom{2^{t+1}B}{2^tB}\binom{2^{t+1}-2^{\ell+2}-\eps}\a\equiv0$ if $0\le\a<2^{t+1}$. This is true if $t>\ell+2$ or ($t=\ell+2$ and $\a\ge0$). If $t=\ell+2$ and $\a<0$, then the binomial coefficient is 0 by consideration of position $2^{\ell+2}$.
\end{proof}

The final case for Theorem \ref{thm1} is
\begin{thm}\label{case5} In the notation of \ref{nk}, and with $\ell=\lg D$, if $B$ is even and $C-2^{\lg C}\ge 2^{\ell+1}$, let $C=2^{\ell+1}A+\g$ with $0\le\g<2^{\ell+1}$, and $m=2^t(B+1)+2^{\ell+1}A-1$. If
\begin{eqnarray}\label{phim}P&=&(V_1\ot1+1\ot V_1)^{m}\cdot\Prod^{C}(V_i\ot1+1\ot V_i)\\
&&\quad \cdot\Prod^{C}(V_i\ot1+1\ot V_i)^2\cdot(R\ot1+1\ot R)^{2n-7-m-3C},\nonumber\end{eqnarray}
then $(\phi_1\ot\phi_2)(P)\ne0\in \zt$.
\end{thm}
\begin{proof} Using the methods of our previous proofs, it suffices to prove that, under the hypotheses, with $\psi(i)=\binom{n-2-i}{k-1-i}$,
the following mod-2 equivalences are valid.
\begin{enumerate}
\item For all $j$ and $g$,
\begin{equation}\label{c1} \tbinom Cj\tbinom Cg\psi(j+g)\psi(2C-j-g+1)\tbinom{2n-7-3C-m}{n-3-j-2g}\equiv0\end{equation}
and
\begin{equation}\label{c2} \tbinom Cj \tbinom Cg\psi(j+g+1)\psi(2C-j-g)\tbinom{2n-7-3C-m}{n-3-j-2g-m}\equiv0.\end{equation}
\item If $\binom Cj\binom Cg\psi(j+g+1)\psi(2C-j-g+1)\equiv1$, then
\begin{enumerate}
\item $\binom{2n-7-3C}{n-3-j-2g}\equiv0$;
\item $\binom{2n-7-3C-m}{n-3-j-2g}\equiv0$;
\item $\binom{2n-7-3C-m}{n-3-j-2g-m}\equiv1$ iff $g=0$ and $j=\g$, in which case $\binom Cj\binom Cg\psi(j+g+1)\psi(2C-j-g+1)\equiv1$.
\end{enumerate}
\end{enumerate}

The proof of (\ref{c1}) and (\ref{c2})  is similar to that for the corresponding terms in the proof of Theorem \ref{case4}. The third and fourth factors will be
of the form $\binom{2^t\a+x}{2^t\b+D}$ and $\binom{2^t\a-1-x}{2^t\b+D}$ with $0< D<2^t$. Their product is 0 mod 2 by the same reasoning as before.

The hypothesis of (2) implies $C-j-g\equiv0$ mod $2^{\ell+1}$ and $|C-j-g|<2^t$, exactly as in the proof of \ref{case4}. Write $C-j-g=2^{\ell+1}K$ with $|2^{\ell+1}K|<2^t$.

Part (2a) is like the first coefficient of (\ref{3b}) except that the constraint on $C-2^{\lg C}$ is different. The argument when $t=\ell+1$ is the same, since the constraint did not occur in that argument. So now assume $t\ge\ell+2$ and write the binomial coefficient as $\ds\binom{2^{t+1}(B+1)-3-C}{2^t(B+1)-2-2^{\ell+2}K-j}$. As before, for both this and $\binom Cj$ to be odd, we must have
$C=2^{\ell+2}Y-\eps$ with $\eps\in\{1,2\}$. We cannot have $C=2^{t+1}-2$ (which implies $D=2^t-1$) because of the assumption that $C-2^{\lg C}\ge 2^{1+\lg D}$. Thus $2^{\ell+2}Y\le 2^{t+1}-2^{\ell+2}$. We have $2^{\ell+1}K=  2^t-2^{\ell+1}-p$ with $p\ge0$, and  then $j\le 2^{\ell+2}Y-\eps-(2^t-2^{\ell+1}-p)$, and hence
$$2^{\ell+2}K+j\le 2^t-2^{\ell+1}+2^{\ell+2}Y-\eps.$$
On the other hand, $2^{\ell+1}K\ge-(2^t-2^{\ell+1})$, and $j\ge-(C-j-g)=-2^{\ell+1}K$, so we have
$$2^{\ell+2}K+j\ge 2^{\ell+1}K\ge -2^t+2^{\ell+1}.$$
Letting $B=2B'$, the binomial coefficient becomes $\ds\binom{2^{t+2}B'+2^{t+1}-2^{\ell+2}Y-3+\eps}{2^{t+1}B'+x}$ with
$$2^{\ell+1}-2^{\ell+2}Y+\eps-2\le x\le 2^{t+1}-2^{\ell+1}-2.$$
If $x\ge0$, this binomial coefficient splits, and is 0 due to $\binom{2^{t+2}B'}{2^{t+1}B'}$. If $x<0$, the binomial coefficient splits as
$$\binom{2^{t+2}B'}{2^{t+1}(B'-1)}\binom{2^{t+1}-2^{\ell+2}Y-3-\eps}{2^{t+1}+x},$$
which is 0 since the second factor has bottom part greater than the top.

To prove (2b), with $C$, $A$, and $\g$ as in the statement of the theorem, the binomial coefficient here is $\binom pq$ with $p=2^t(B+1)-2-C-2^{\ell+1}A$ and $q=2^t(B+1)-1+C-j-2g$. Then $q-p=2(C-j-g)+j+2^{\ell+1}A+1$. Since $C-j-g\ge -C$ and is a multiple of $2^{\ell+1}$, $C-j-g\ge-2^{\ell+1}A$.
Similarly to the previous case, this implies $2(C-j-g)+j\ge -2^{\ell+1}A$. Thus $q-p>0$ and $\binom pq=0$.

Finally, for (2c), the binomial coefficient becomes $\binom{2^t(B+1)-2-C-2^{\ell+1}A}{\g-j-2g}$.
For this to be nonzero, we must have $j+2g\le\g$. But $j+g\equiv\g$ mod $2^{\ell+1}$, and so we must have $g=0$ and $j=\g$, in which case the binomial coefficient equals 1. Clearly $\binom Cj\binom Cg\equiv1$. Also, $\psi(j+g+1)=\psi(\g+1)=\binom{2^t(B+1)+2^{\ell+1}A-1}{2^tB+D}\equiv1$ and $\psi(2C-j-g+1)=\psi(2C-\g+1)=\binom{2^t(B+1)-2^{\ell+1}A-1}{2^tB+D}\equiv1$.
These use the fact that, since $D\le 2^{\ell+1}-1$, $C\le2^t+2^{\ell+1}-2$, and hence $2^{\ell+1}A\le 2^t$.

\end{proof}
We close by showing that all $(2n-6)$-fold products $P$ of elements of the form $(y\ot1+1\ot y)$ in $H^*(\Mbar_{n,n-2k}\times \Mbar_{n,n-2k})$ are zero. This will complete the proof of Theorem \ref{thm3}. First note that all such products are invariant under the involution that interchanges factors. If $m_1$ and $m_2$ are monomials of degree $n-3 $  in the $V_i$'s and $R$, then $m_1\ot m_2+m_2\ot m_1=0$ since $m_i$ equals either 0 or the unique nonzero class. Thus it suffices to show that the coefficient of any $T_{S,n-3}\ot T_{S,n-3}$ in $P$ is 0. We prove this by induction on $|S|$.
Note that the factors which we must consider are not just those of the form $(V_i\ot1+1\ot V_i)$ and $(R\ot1+1\ot R)$, but also sums of these.

The coefficient of $R^{n-3}\ot R^{n-3}$ in $P$ is 0 if $P$ contains any factors which do not contain terms $(R\ot1+1\ot R)$, while if all factors contain such terms, it is $\binom{2n-6}{n-3}\equiv0$. This initiates the induction, as $|S|=0$ here. Assume that all terms $T_{S,n-3}\ot T_{S,n-3}$ in any product $P$ are 0 if $|S|<s$. Let $S$ be a subset with $|S|=s$.
In $P$, we may omit all terms $(V_i\ot1+1\ot V_i)$ for which $i\not\in S$. If this omission makes any of the factors become 0, then the coefficient of $T_{S,n-3}\ot T_{S,n-3}$ is 0.
Otherwise, by the induction hypothesis, the coefficient of all $T_{S',n-3}\ot T_{S',n-3}$ with $S'$ a proper subset of $S$ is 0, and since the sum of all coefficients in $H^{n-3}(\Mbar_{n,n-2k})\ot H^{n-4}(\Mbar_{n,n-2k})$ is $\binom{2n-6}{n-3}$, which is even,
 the coefficient of the remaining term $T_{S,n-3}\ot T_{S,n-3}$ must also be 0.
 \def\line{\rule{.6in}{.6pt}}

\end{document}